\documentclass{birkjour}

\usepackage{algorithmic}
\usepackage{algorithm}

 \usepackage{amsmath}
 \usepackage{amssymb}
 \usepackage{amsthm}
 \usepackage{mathrsfs}

\def\R		{\mathbb{R}}

\def\N		{\mathbb{N}}
\def\O		{\mathcal{O}}
\def\d		{\:\mathrm{d}}

\def\span	{\operatorname{span}}

\def\sgn	{\operatorname{sgn}}

\def\L		{\mathscr{L}}

\def\F		{\mathscr{F}}
\renewcommand{\vec}[1]{\boldsymbol{#1}}

\def\G		{\mathcal{G}}
\def\j		{\vec{j}}
\def\k		{\vec{k}}
\def\l		{\vec{l}}
\def\e		{\vec{e}}

\def\a		{\vec{a}}
\def\b		{\vec{b}}

\def\A		{\vec{A}}
\def\B		{\vec{B}}
\def\x		{\vec{x}}
\def\y		{\vec{y}}
\def\z		{\vec{z}}
\def\u		{\vec{u}}
\def\v		{\vec{v}}
\def\w		{\vec{w}}

\def\P		{\mathcal{P}}

 \newtheorem{thm}{Theorem}[section]
 \newtheorem{cor}[thm]{Corollary}
 \newtheorem{lem}[thm]{Lemma}
 
 \theoremstyle{definition}
 \newtheorem{defn}[thm]{Definition}
 \theoremstyle{remark}
 \newtheorem{rem}[thm]{Remark}
 \newtheorem*{ex}{Example}
 \numberwithin{equation}{section}

  \newtheorem{bez}[thm]{Notation}

\begin{document}

\title[General GFT Convolution Thorem]{A General Geometric Fourier Transform \\Convolution Theorem}

\author[Bujack]{Roxana Bujack}
\address{%
Universit\"at Leipzig\\
Institut f\"ur Informatik\\
Johannisgasse 26\\
04103 Leipzig\\
Germany\\
}
\email{bujack@informatik.uni-leipzig.de}

\author[Scheuermann]{Gerik Scheuermann}
\address{%
Universit\"at Leipzig\\
Institut f\"ur Informatik\\
Johannisgasse 26\\
04103 Leipzig\\
Germany\\
}
\email{scheuermann@informatik.uni-leipzig.de}

\author[Hitzer]{Eckhard Hitzer}
\address{%
University of Fukui\\
Department of Applied Physics\\
3-9-1 Bunkyo\\
Fukui 910-8507\\
Japan\\
}
\email{hitzer@mech.u-fukui.ac.jp}

\date{\today}

\begin{abstract}
The large variety of Fourier transforms in geometric algebras inspired the straight forward definition of ``A General Geometric Fourier Transform`` in Bujack et al., Proc. of ICCA9, covering most versions in the literature. We showed which constraints are additionally necessary to obtain certain features like linearity, a scaling, or a shift theorem. In this paper we extend the former results by a convolution theorem.
\end{abstract}
\keywords{Fourier transform, geometric algebra, Clifford algebra, convolution, coorthogonal, geometric trigonometric transform.}

\maketitle
\section{Introduction}
The Fourier transform (FT) is a very important tool for mathematics, physics, computer science and engineering. Since geometric algebras \cite{C1878} usually contain continuous submanifolds of geometric square roots of minus one \cite{HA10,HHA11} there are infinitely many ways to construct new geometric Fourier transforms by replacing the imaginary unit in the classical definition of the FT. Every multivector comes with a natural geometric interpretation so the generalization is very useful. It helps to interpret the transform and apply it in a target oriented way to the specific underlying problem.
\par
Many different definitions of Fourier transforms in geometric algebras were developed. For example the Clifford Fourier transform introduced by Jancewicz \cite{Janc90} and expanded by Ebling and Scheuermann \cite{Ebl06} and Hitzer and Mawardi \cite{HM08}, the one established by Sommen in \cite{Som82} and re-established by B\"ulow \cite{Bue99}, the quaternionic Fourier transform by Ell \cite{Ell93} and later by B\"ulow \cite{Bue99}, the spacetime Fourier transform by Hitzer \cite{Hitz07}, the Clifford Fourier transform for color images by Batard et al. \cite{BBS08}, the Cylindrical Fourier transform by Brackx et al. \cite{BSS10}, the transforms by Felsberg \cite{Fels02} or Ell and Sangwine \cite{ES00,ES07}.
\par
We abstracted all of them in one general definition in \cite{BSH11}. There we analyzed the separation of constant factors from the transform, the scaling theorem and shift properties. Now we want to derive a convolution theorem. Although this paper is written to be self-contained we highly recommend to read the preceding work. Lemmata that were introduced there and will be needed again will be repeated, but the study of the proofs, that can be found in \cite{BSH11}, will grant a deeper understanding of the geometric context.
\par
We examine geometric algebras $\G^{p,q},p+q=n\in\N$ over $\R^{p,q}$ \cite{HS84} generated by the associative, bilinear geometric product with neutral element $1$ satisfying 
\begin{equation}
\begin{aligned}
 \e_j\e_k+\e_k\e_j=\epsilon_j\delta_{jk},
\end{aligned}
\end{equation}
for all $j,k\in\{1,...,n\}$ with the Kronecker symbol $\delta$ and 
\begin{equation}\begin{aligned}
\epsilon_j=\begin{cases}1&\forall j=1,...,p,\\-1&\forall j=p+1,...,n.\end{cases} 
\end{aligned}\end{equation}
For the sake of brevity we want to refer to arbitrary multivectors 
\begin{equation}\begin{aligned}\label{mv}
 \A=\sum\limits_{k=0}^n\sum\limits_{1\leq j_1<...<j_k\leq n}a_{j_1...j_k}\e_{j_1}...\e_{j_k}\in\G^{p,q},
\end{aligned}\end{equation}
$ a_{j_1...j_k} \in\R,$ as 
\begin{equation}\label{mimv}
 \A=\sum_{\j}a_{\j}\e_{\j}.
\end{equation}
where each of the $2^n$ multi-indices $\vec{j}\subseteq\{1,...,n\}$ indicates a basis multivector of dimension $k$ of $\G^{p,q}$ by $\e_{\j}=\e_{j_1}...\e_{j_{k}},$ $1\leq j_1<...<j_{k}\leq n,\e_{\emptyset}=\e_0=1$ and its associated coefficient $a_{\j}=a_{j_1...j_{k}}\in\R$. For each geometric algebra $\G^{p,q}$ we will write $\mathscr I^{p,q}=\{i\in\G^{p,q},i^2\in\R^-\}$ to denote the real multiples of all geometric square roots of minus one, compare \cite{HA10} and \cite{HHA11}. We chose the symbol $\mathscr I$ to be reminiscent of the imaginary numbers. 
\par
Throughout this paper we analyze multivector fields $\A:\R^{p^\prime,q^\prime}\to\G^{p,q},$ $p^\prime+q^\prime=m\in\N,p+q=n\in\N$. To keep notations short we will often denote the argument vector space by just $\R^m$, so please keep in mind, that it has a signature $p^\prime,q^\prime$, too.
\par
We defined the general \textbf{Geometric Fourier Transform} (GFT) $\F_{F_1,F_2}(\A)$ of a multivector field $\A:\R^{p^\prime,q^\prime}\to\G^{p,q},p^\prime+q^\prime=m\in\N,p+q=n\in\N$ in \cite{BSH11} by the calculation rule
\begin{equation}\begin{aligned}\label{gft}
 \F_{F_1,F_2}(\A)(\u):=\int_{\R^m}\prod_{f\in F_1}e^{-f(\x,\u)}\A(\x)\prod_{f\in F_2}e^{-f(\x,\u)}\d^m \x,
\end{aligned}\end{equation}
with $\x,\u\in\R^m$ and two ordered finite sets $F_1=\{f_1(\x,\u),...,f_{\mu}(\x,\u)\},$ $F_2=\{f_{\mu+1}(\x,\u),...,f_{\nu}(\x,\u)\}$ of mappings $f_l(\x,\u):\R^m\times \R^m\to\mathscr I^{p,q},\forall l=1,...,\nu$.
We proved some fundamental theorems in dependence on properties of the functions $f_l$, like existence, linearity, shift and scaling.
%
\section{Coorthogonality and Bases}
%
\begin{defn}
We call two vectors $\v,\w$ \textbf{orthogonal} ($\v\perp\w$) if $\v\cdot\w=0$ and \textbf{colinear} ($\v\parallel\w$) if $\v\wedge\w=0$.
\end{defn}
\begin{defn}\label{d:orthcolin2}
We call two blades $\A,\B$ \textbf{orthogonal} ($\A\perp\B$) if all of their generating vectors are mutually orthogonal and  \textbf{colinear} ($\A\parallel\B$) if all vectors from one blade are colinear to all vectors in the other one.
\end{defn}
For a vector $\v$ and a blade $\B=\b_1\wedge...\wedge \b_d$ the following equalities hold
\begin{equation}\begin{aligned}\label{coorth}
\v\perp\B\Leftrightarrow\v\B=\v\wedge\B\Leftrightarrow\v\cdot\B=0\Leftrightarrow\v\B=(-1)^{d}\B\v,\\
\v\parallel\B\Leftrightarrow\v\B=\v\cdot\B\Leftrightarrow\v\wedge\B=0\Leftrightarrow\v\B=(-1)^{d-1}\B\v,
\end{aligned}\end{equation}
compare \cite{HS84,Hes86}. That inspires the next definition.
\begin{defn}
We call two blades $\A$ and $\B$ \textbf{coorthogonal} if $\A\B=\pm \B\A$.
\end{defn}
\begin{bez}
A blade can alternatively be written as an outer product of vectors or as a geometric product of orthogonal vectors. For blades $\A=\a_1\wedge...\wedge \a_\mu$ and $\B=\b_1\wedge...\wedge \b_\nu$ we will use the notations $\span(\B):=\span(\b_1,...,\b_\nu)$, $\A\oplus\B:=\span(\A)\oplus\span(\B)\subseteq\R^{p,q}$, $\A\cap\B:=\span(\A)\cap\span(\B)\subseteq\R^{p,q}$, $\beta(\A,\B):=\dim(\A\cap\B)$ and $\alpha(\A,\B):=\dim(\A\oplus\B)=\mu+\nu-\beta(\A,\B)$. For a set of blades $B=\{\B_1,...,\B_d\},d\in\N$ we use the notation $\span(B)=\bigoplus_{k=1}^d\span (\B_k)$ and  $\alpha(B)=\dim(\span(B))$.
\end{bez}
\begin{lem}\label{l:orth_basis1}
The basis blades $\e_{\k}$ of $\G^{p,q}$ that are generated from an orthogonal basis of $\R^{p,q}$ are mutually coorthogonal.
\end{lem}
\begin{proof}
 All orthogonal basis vectors of $\R^{p,q}$ satisfy
\begin{equation}
\e_j\e_k=\begin{cases}-\e_k\e_j,&\text{for }j\neq k\in\N\\ \e_k\e_j,&\text{for }j= k\end{cases}
\end{equation}
in every geometric algebra $\G^{p,q}$. So for two basis blades $\e_{\j}=\e_{j_1,...,j_{\mu}},\e_{\k}=\e_{k_1,...,k_{\nu}},1\leq j_1<...<j_{\mu}\leq n,1\leq k_1<...<k_{\nu}\leq n$ with dimensions $\mu$, respectively $\nu$ we get
\begin{equation}\label{acom}
 \e_{\j}\e_{\k}=(-1)^{\mu\nu-\beta(\e_{\j},\e_{\k})}\e_{\k}\e_{\j},
\end{equation}
where $\beta(\e_{\j},\e_{\k})=|\{l\in\N,l\in\j \text{ and } l\in\k\}|$ is the number of indices appearing in both sets, respectively the dimension of the meet of the two blades. 
\end{proof}
%
%
\begin{lem}\label{l:orth_basis2}
 For two coorthogonal blades $\A$ and $\B$ there is an orthonormal basis $V=\{\v_1,...,\v_{\alpha(\A,\B)}\}$ of $\A\oplus\B\subseteq\R^{p,q}$ such that both can be expressed as real multiples of basis blades, that means $\A=\sgn(\A)|\A|\v_{j_1}...\v_{j_\mu}$ and $\B=\sgn(\B)|\B|\v_{k_1}...\v_{k_\nu}, a,b\in\R$  with the signum function being $1$ or $-1$.
\end{lem}
\begin{proof}
We know from \cite{HS84} that every blade $\A$ spans a vector space $\span(\A)$, that this vector space has an orthonormal basis, how it can be produced and that for $\span(\B)\subset \span(\A)$ there is a unique blade $\A_{\perp \B} :=\B^{-1}\cdot \A$ orthogonal to $\B$, such that $\span(\B)\cap \span(\A_{\perp \B})=\emptyset,\span(\B)\cup \span(\A_{\perp \B})=\span(\A)$.
\par
Therefore we can separate the space $\A\oplus\B=(\A\cap\B)\cup\span(\A_{\perp (\A\cap\B)})\cup\span(\B_{\perp (\A\cap\B)})$ into three disjoint parts and immediately know that $\A\cap\B$ is orthogonal to both $\A_{\perp (\A\cap\B)}$ and $\B_{\perp (\A\cap\B)}$. 
Since $\span(\A_{\perp (\A\cap\B)})\cap\span(\B_{\perp (\A\cap\B)})=\emptyset$ all basis vectors $\a_1,...,\a_{\mu-\beta(\A,\B)}$ of $\A_{\perp (\A\cap\B)}$ satisfy $\a_1,...,\a_{\mu-\beta(\A,\B)} \notin\span(\B_{\perp (\A\cap\B)})$ so $\forall j=1,...,\mu-\beta(\A,\B):\a_j\wedge\B_{\perp (\A\cap\B)}\neq0$. For the product of a vector and a blade we have
\begin{equation}\begin{aligned}\label{unabh->orth}
\vec a_j\B_{\perp (\A\cap\B)}=&\vec a_j\cdot \B_{\perp (\A\cap\B)}+\vec a_j\wedge \B_{\perp (\A\cap\B)}
\\=& (-1)^{\nu-\beta(\A,\B)-1}\B_{\perp (\A\cap\B)}\cdot\vec a_j+(-1)^{\nu-\beta(\A,\B)} \B_{\perp (\A\cap\B)}\wedge\vec a_j
\end{aligned}\end{equation} 
and since $\a_j\wedge\B\neq0$ necessarily $\vec a_j\cdot\vec \B=0$ has to be valid $\forall j=1,...,\mu-\beta(\A,\B)$ in order to satisfy coorthogonality. That is equivalent to 
$\A_{\perp (\A\cap\B)}\perp\B_{\perp (\A\cap\B)}$ and therefore unifying the orthonormal bases of all three parts form an orthonormal basis of $\A\oplus\B$. Let $\b_1,...,\b_{\nu-\beta(\A,\B)}$ be the basis of $\span(\B)$, $\vec c_1,...,\vec c_{\beta(\A,\B)}$ be the basis of $\A\cap\B$ then the blade $\A$ has can be written as $\A=\sgn(\A)|\A|\vec c_1,...,\vec c_{\beta(\A,\B)},\a_1,...,\a_{\mu-\beta(\A,\B)}$ and the blade $\B$ as $\B=\sgn(\B)|\B|\vec c_1,...,\vec c_{\beta(\A,\B)},\b_1,...,\b_{\nu-\beta(\A,\B)}$.
\end{proof}
\begin{rem}
An alternative proof can be achieved from looking at the geometric product $\A\B$ as detailed in \cite{Hitz10a} equation (45) and more generally in \cite{Hitz10b} equation (17). There the lowest order term is a product of all the cosines of the principal angles and the lowest +2 order term is a sum of summands each with the product of one sine of one principal angle times the cosines of all other principal angels. The transition from $\A\B$ to $\B\A$ then does not change the sign of the lowest order term, but changes the signs of all summands in the lowest +2 order grade part. Coorthogonality is therefore only possible if for every principal angle either the cosine is zero (the 2 associated principal vectors are perpendicular), or if the sine is zero (the 2
associated principal vectors are parallel). Hence all principal angles are in $\{0, \frac\pi2\}$.
\end{rem}
\begin{lem}\label{l:orth_basis3}
 Let $B=\{\B_1,...,\B_d\},d\in\N$ be non-zero mutually coorthogonal blades. Then there is an orthonormal basis $\v_1,...,\v_{\alpha(B)}$ of $\span(B)$ such that every $\B_k,k=1,...,d$ can be written as a real multiple of a basis blade, that means $\B_k=\sgn(\B_k)|\B_k|\v_{\j(k)}$ with $\v_{\j(k)}=\v_{j_1(k),...,j_{\mu}(k)}, \mu=\mu(k)=\dim(\B_k), |\B_k|\in\R$.
\end{lem}
\begin{proof}
Algorithm 1 constructs this basis. 
\begin{algorithm}
\caption{Construction of the basis}
\label{alg1}
\begin{algorithmic}[1]
\REQUIRE $B,d$
\STATE $C=B$, $Basis=\emptyset$, $l=0$, $\forall k=1,...,d:\j(k)=(\emptyset)$,
\WHILE{$l< \alpha(C)$}
 \STATE $D=\bigcup_{M\subseteq C}(\bigcap_{C_k\in M}\span(\vec C_k))$
 \STATE choose $\vec D_k$ of minimal dimension in $D$, choose $\vec c\in\span(\vec D_k)$, 
 \STATE $Basis=Basis\cup \vec c$, $l=l+1$,
 \FOR{$k=1,...,d$}
  \IF{$\vec c\in\span(\vec C_k)$}
    \STATE $\vec C_k=\vec c^{-1}\cdot\vec C_k$, $\j(k)=(\j(k),l)$,
  \ENDIF
 \ENDFOR
\ENDWHILE
\ENSURE $Basis,\j(k),\forall k=1,...,d$
\end{algorithmic}
\end{algorithm}
\par
The set $D$ is the set of intersections of the subspaces spanned by all possible combinations of elements of $C$. 
The elements $\vec D_k$ of $D$ with minimal dimension satisfy $\forall \vec D_j\in D:\vec D_k\cap \vec D_j\in\{\emptyset,\vec D_k\}$, because otherwise $\vec D_k\cap \vec D_j$ would have lower dimension than $\vec D_k$ which is a contradiction. 
In both cases all generating vectors of $\vec D_k$ can be added to the basis, compare the proof of Lemma \ref{l:orth_basis2}. So the choice of any vector $\vec c\in \span(\vec D_k)$ will be successful.
\par
Once a vector is chosen there are two more cases that already appeared in the proof of Lemma \ref{l:orth_basis2}. In the case of $\vec c\notin\span(\vec C_k)$ follows, that $\vec c$ is orthogonal to $\vec C_k$, because of (\ref{unabh->orth}). 
In the case of $\vec c\in\span(\vec C_k)$ the multiplication of $\vec c$ to $\vec C_k$ in the algorithm always creates blades $\vec c\vec C_k$ of lower dimension orthogonal to $\vec c$, because of
 \begin{equation}\begin{aligned}\label{dim_senk}
 \vec c^{-1}\cdot \B=&\langle \vec c^{-1}\B\rangle_{\dim(\B)-1}.
 \end{aligned}\end{equation} 
Therefore the application of this operation to all blades in $C$ only leaves blades that are orthogonal to $\vec c$ but still coorthogonal amongst each other because of
 \begin{equation}\begin{aligned}\label{coorth_bleibt}
 \vec c^{-1}\vec C_j\vec c^{-1}\vec C_k\overset{(\ref{coorth})} =&(-1)^\mu\vec C_j\vec c^{-1}\vec c^{-1}\vec C_k\\
 \overset{(\vec c^{-1})^2\in\R}=&(-1)^\mu(\vec c^{-1})^2\vec C_j\vec C_k\\
 \overset{\text{prereq}}=&\pm(-1)^\mu(\vec c^{-1})^2\vec C_k\vec C_j\\
\overset{(\ref{coorth})}=&\pm(-1)^{\mu+\nu-1}\vec c^{-1}\vec C_k\vec c^{-1}\vec C_j.
 \end{aligned}\end{equation}
At the beginning $C$ spans the whole space $\span(C)=\span(B)$ but as the algorithm proceeds $\span(C)=\span(B)\setminus\span(Basis)$ such that $\span(C)$ and $\span(Basis)$ are orthogonal. Because of that the set $Basis$ is orthogonal at all times. The algorithm stops when $\alpha(C)$ vectors are in $Basis$.
So finally $\alpha(C)$ orthogonal vectors will be in $Basis$, which therefore in deed is an orthogonal basis of $\span(B)$, all elements of $C$ will have dimension zero and the algorithm will end returning the basis and how the blades can be constructed.
\end{proof}
So trivially spoken, coorthogonality of blades can as well be interpreted as coorthogonality of all their generating vectors, that means all their generating vectors are either orthogonal or colinear.
\begin{thm}\label{t:orth_basis}
 A finite number of blades are coorthogonal if and only if they are real multiples of basis blades of an orthonormal basis of $\R^{p,q}$.
\end{thm}
\begin{proof}
 The assertion follows from Lemma \ref{l:orth_basis1} and \ref{l:orth_basis3} together with normalization, the basis completion theorem and the Gram-Schmidt orthogonalization process.
\end{proof}
%
%
\begin{bez}\label{b:coorth}
Throughout this paper we will only deal with geometric Fourier transforms whose defining functions $f_1,...,f_{\nu}$, compare (\ref{gft}), are mutually coorthogonal 
blades, that means they satisfy the property $\forall l,k=1,...,\nu,\forall \x,\u\in\R^m:$
\begin{equation}\begin{aligned}
f_l(\x,\u)f_k(\x,\u)=\pm f_k(\x,\u)f_l(\x,\u).
\end{aligned}\end{equation}
Theorem \ref{t:orth_basis} allows us to write 
 \begin{equation}\begin{aligned}\label{bed}
 f_l(\x,\u)= \sgn(f_l(\x,\u))|f_l(\x,\u)|\e_{\j_l(\x,\u)}.
 \end{aligned}\end{equation}
 for all $l=1,...,\nu$ with a real valued function $|f_l(\x,\u)|:\R^m\times \R^m\to\R$ and a function ${\j_l(\x,\u)}:\R^m\times \R^m\to\P(\{1,...,n\})$ that maps to a multi-index indicating a basis multivector of a certain basis. We will refer to a set of functions with this property simply as a set of basis blade functions.
\end{bez}

\begin{ex}\label{b:gft}
This constraint seems strong but all standard examples of geometric Fourier transforms from \cite{BSH11} fulfill it.
\begin{enumerate}
\item For $\A:\R^n\to\G^{n,0}, n=2\pmod 4$ or $n=3\pmod 4$, the Clifford Fourier transform introduced by Jancewicz \cite{Janc90} for $n=3$ and expanded by Ebling and Scheuermann \cite{Ebl06} for $n=2$ and Hitzer and Mawardi \cite{HM08} for $n=2\pmod 4$ or $n=3\pmod 4$ with
\begin{equation}\begin{aligned}
F_1=&\emptyset,\\F_2=&\{f_1\},\\f_1(\x,\u)=&2\pi i_n \x\cdot\u,
\end{aligned}\end{equation}
clearly fulfills the restriction, since it has only one defining function and $i_n$ is a basis blade.
\item The Sommen B\"ulow Clifford Fourier transform from \cite{Som82,Bue99}, defined by
\begin{equation}\begin{aligned}
F_1=&\emptyset,\\F_2=&\{f_1,...,f_n\},\\f_l(\x,\u)=&2\pi \e_lx_lu_l,\forall l=1,...,n,
\end{aligned}\end{equation}
for multivector fields $\R^n\to\G^{0,n}$ fulfills it, because all basis vectors $\e_k$ are of course basis blades.
\item For $\A:\R^2\to\G^{0,2}\approx\mathbb{H}$ the quaternionic Fourier transform \cite{Ell93,Bue99} is generated by 
\begin{equation}\begin{aligned}
F_1=&\{f_1\},\\F_2=&\{f_2\},\\f_1(\x,\u)=&2\pi ix_1u_1,\\f_2(\x,\u)=&2\pi jx_2u_2,
\end{aligned}\end{equation}
and satisfies the condition because $i$ and $j$ are basis blades, too.
\item The defining functions of the spacetime Fourier transform by Hitzer \cite{Hitz07}\footnote{Please note that Hitzer uses a different notation in \cite{Hitz07}. His $\x=t\e_0+x_1\e_1+x_2\e_2+x_3\e_3$ corresponds to our $\x=x_1\e_1+x_2\e_2+x_3\e_3+x_4\e_4$, with $\e_0\e_0=\epsilon_0=-1$ being equivalent to our $\e_4\e_4=\epsilon_4=-1$.} with the $\G^{3,1}$-pseudoscalar $i_4$ and
\begin{equation}\begin{aligned}
F_1=&\{f_1\},\\F_2=&\{f_2\},\\f_1(\x,\u)=& \e_4x_4u_4,\\f_2(\x,\u)=& \epsilon_4\e_4i_4(x_1u_1+x_2u_2+x_3u_3),
\end{aligned}\end{equation}
fulfill coorthogonality of blades, because of $\e_4\parallel i_4\Rightarrow\e_4\perp \e_4i_4$.
\item The Clifford Fourier transform for color images by Batard, Berthier and Saint-Jean \cite{BBS08} for $m=2,n=4,\A:\R^2\to\G^{4,0}$, a fixed bivector $\B$, and the pseudoscalar $i$ can be written as 
\begin{equation}\begin{aligned}
F_1=&\{f_1,f_2\},\\F_2=&\{f_3,f_4\},\\f_1(\x,\u)=&\frac12 (x_1u_1+x_2u_2)\B,\\f_2(\x,\u)=&\frac12 (x_1u_1+x_2u_2)i\B,\\f_3(\x,\u)=&-f_1(\x,\u)
,\\f_4(\x,\u)=&-f_2(\x,\u).
\end{aligned}\end{equation}
There are bivectors in $\G^{4,0}$ that are not blades. But since Batard et al. start from $\G^{3,0}$ we may assume $\B$ to be a blade. So the transform fulfills condition (\ref{bed}), because $\B$ and $i\B$ commute. Let $\B$ consist of the two orthogonal vectors $\v_1\v_2=\B$, then a basis as in Theorem \ref{t:orth_basis} could be constructed by orthogonal basis completion of $\v_1\v_2$ to a basis $B=\{\v_1,\v_2,\v_3,\v_4\}$ and normalization. Because from $\v_1\v_2\v_3\v_4=ci, c\in\R$ follows that $i\B=-c^{-1}\v_3\v_4$ is a basis blade, too.
%
%
\item The cylindrical Fourier transform as introduced by Brackx, De Schepper and Sommen in \cite{BSS10} with
\begin{equation}\begin{aligned}
F_1=&\{f_1\},\\F_2=&\emptyset,\\f_1(\x,\u)=&{-\x\wedge\u},
\end{aligned}\end{equation}
satisfies the restriction because it has only one defining function, too. We will see that in contrast to the other transforms the basis guaranteed by Theorem \ref{t:orth_basis} depends locally on $\x$ and $\u$ here.
\end{enumerate}
\end{ex}
\begin{rem}
Theorem \ref{t:orth_basis} guarantees, that there is an orthonormal basis of $\R^{p,q}$ such that $\forall l=1,...,\nu,\forall \x,\u\in\R^m:$ the values of the functions $f_l(\x,\u)=\sgn(f_l(\x,\u))|f_l(\x,\u)|\e_{\k(l)}$ are real multiples of basis blades of $\G^{p,q}$. We assume that this basis is the one we use and we call the basis vectors simply $\e_1,...,\e_n$. Therefore we can use the terms coorthogonal blades and basis blades as synonyms up to a real multiple, especially in terms of commutativity properties they can be used equivalently.
\end{rem}
\section{Products with Basis Blades}
%
From \cite{BSH11} we already know the following facts about products with invertible multivectors. Please note that 
every square root of minus one $i\in\mathscr I^{p,q}$ 
is invertible and that 
therefore the functions $f_l:\R^m\times \R^m\to\mathscr I^{p,q}$
 from (\ref{gft}) are pointwise invertible, too.
\begin{defn}\label{d:c}
For an invertible multivector $\B\in\G^{p,q}$ and an arbitrary multivector $\A\in\G^{p,q}$ we define
\begin{equation}\begin{aligned}
 \A_{\vec c^0(\B)}=&\frac12(\A+\B^{-1}\A\B),\\
\A_{\vec c^1(\B)}=&\frac12(\A-\B^{-1}\A\B).
\end{aligned}\end{equation}
\end{defn}
\begin{defn}\label{d:c^j}
 For $d\in\N,\A\in\G^{p,q}$, the ordered set $B=\{\B_1,...,\B_d\}$ of invertible multivectors of $\G^{p,q}$ and any multi-index $\j\in\{0,1\}^{d}$ we define
\begin{equation}\begin{aligned}
 \A_{\vec c^{\j}(\overrightarrow B)}:=&((\A_{\vec c^{j_1}(\B_1)})_{\vec c^{j_2}(\B_2)}...)_{\vec c^{j_d}(\B_d)},\\
 \A_{\vec c^{\j}(\overleftarrow B)}:=&((\A_{\vec c^{j_d}(\B_d)})_{\vec c^{j_{d-1}}(\B_{d-1})}...)_{\vec c^{j_1}(\B_1)}
\end{aligned}\end{equation}
recursively with $\vec c^0, \vec c^1$ of Definition \ref{d:c}.
\end{defn}
\begin{lem}\label{l:prodviele}
 Let $d\in\N,B=\{\B_1,...,\B_d\}$ be invertible multivectors and for $\j\in\{0,1\}^{d}$ let $|\j|:=\sum_{k=1}^dj_k$, then $\forall\A\in\G^{p,q}$
\begin{equation}\begin{aligned}
\A=&\sum_{\j\in\{0,1\}^{d}}\A_{\vec c^{\j}(\overrightarrow B)},\\
\A\B_1...\B_d=&\B_1...\B_d\sum_{\j\in\{0,1\}^{d}}(-1)^{|\j|}\A_{\vec c^{\j}(\overrightarrow B)},\\
\B_1...\B_d\A=&\sum_{\j\in\{0,1\}^{d}}(-1)^{|\j|}\A_{\vec c^{\j}(\overleftarrow B)}\B_1...\B_d.
\end{aligned}\end{equation}
\end{lem}
%
Now we use the concept of coorthogonality to simplify and enhance the preliminary findings. For $d\in\N$ we take a closer look at sets of coorthogonal blades $B=\{\B_1,...,\B_d\}$. 
%
\begin{lem}
 Let $B=\{\B_1,...,\B_d\},d\in\N$ be a set of mutually coorthogonal blades with the unique inverse $\B_k^{-1}={\B_k}{\B_k^{-2}},\B_k^2\in\R\setminus\{0\}$. Further let $\A\in\G^{p,q}$ and $\j\in\{0,1\}^{d}$ be arbitrary, then $\A_{\vec c^{\j}(\overrightarrow B)}$ and $\A_{\vec c^{\j}(\overleftarrow B)}$ are independent from the order of $B$.
\end{lem}
\begin{proof}
For two blades $\B_k,\B_l$ we have 
 \begin{equation}\begin{aligned}\label{bv_umkehren1}
 \A_{\vec c^{j_k,j_l}(\overrightarrow{\B_{k}\B_l})}=&((\A_{\vec c^{j_k}(\B_k)})_{\vec c^{j_l}(\B_l)}\\
=&\frac 12(\A+(-1)^{j_k}\B_k^{-1}\A\B_k)_{\vec c^{j_l}(\B_l)}\\
=&\frac 12(\frac 12((\A+(-1)^{j_k}\B_k^{-1}\A\B_k)
\\&+(-1)^{j_l}\B_l^{-1}(\A+(-1)^{j_k}\B_k^{-1}\A\B_k)\B_l)\\
=&\frac 14(\A+(-1)^{j_k}\B_k^{-1}\A\B_k+(-1)^{j_l}\B_l^{-1}\A\B_l
\\&+(-1)^{j_k+j_l}\B_l^{-1}\B_k^{-1}\A\B_k\B_l).
\end{aligned}\end{equation}
From the prerequisites follows
\begin{equation}\begin{aligned}
\B_k^{-1}\B_l^{-1}\A\B_l\B_k=&\frac{\B_k\B_l\A\B_l\B_k}{\B_k^{2}\B_l^{2}}\\
=&\frac{\pm\B_l\B_k\A(\pm)\B_k\B_l}{\B_k^{2}\B_l^{2}}\\
=&\B_l^{-1}\B_k^{-1}\A\B_k\B_l,
\end{aligned}\end{equation}
which inserted into (\ref{bv_umkehren1}) leads to
 \begin{equation}\begin{aligned}\label{bv_umkehren2}
 \A_{\vec c^{j_k,j_l}(\overrightarrow{\B_{k}\B_l})}=&\frac 14(\A+(-1)^{j_k}\B_k^{-1}\A\B_k+(-1)^{j_l}\B_l^{-1}\A\B_l
\\&+(-1)^{j_k+j_l}\B_k^{-1}\B_l^{-1}\A\B_l\B_k)\\
=&((\A_{\vec c^{j_l}(\B_l)})_{\vec c^{j_k}(\B_k)}\\
=&\A_{\vec c^{j_k,j_l}(\overleftarrow{\B_k\B_l})}.
\end{aligned}\end{equation}
Since $((\A_{\vec c^{j_1}(\B_1)})_{\vec c^{j_2}(\B_2)}...)_{\vec c^{j_k}(\B_k)}$ is a multivector, the application of (\ref{bv_umkehren2}) leads to 
 \begin{equation}\begin{aligned}
 \A_{\vec c^{\j}(\overrightarrow B)}=&((\A_{\vec c^{j_1}(\B_1)})_{\vec c^{j_2}(\B_2)}...)_{\vec c^{j_k}(\B_k)})_{\vec c^{j_{k+1}}(\B_{k+1})})...)_{\vec c^{j_d}(\B_d)}\\
=&((\A_{\vec c^{j_1}(\B_1)})_{\vec c^{j_2}(\B_2)}...)_{\vec c^{j_{k+1}}(\B_{k+1})})_{\vec c^{j_{k}}(\B_{k})})...)_{\vec c^{j_d}(\B_d)}\\
=&\A_{\vec c^{\j}(\overrightarrow {\B_1,...,\B_{k+1},\B_{k},...,\B_d})},
\end{aligned}\end{equation}
that means no transposition of two neighboring multivectors changes the value of $\A_{\vec c^{\j}(\overrightarrow B)}$. The assertion follows because every permutation can be constructed from the composition of these swaps.
\end{proof}
\begin{cor}\label{k:r=l}
 For $d\in\N,\A\in\G^{p,q}$, the ordered set $B=\{\B_1,...,\B_d\}$ of mutually coorthogonal blades and any multi-index $\j\in\{0,1\}^{d}$ we have
\begin{equation}\begin{aligned}
 \A_{\vec c^{\j}(\overrightarrow B)}=\A_{\vec c^{\j}(\overleftarrow B)}.
\end{aligned}\end{equation}
\end{cor}
\begin{bez}
Because of Corollary \ref{k:r=l} we will not distinguish between $\A_{\vec c^{\j}(\overrightarrow B)}$ and $\A_{\vec c^{\j}(\overleftarrow B)}$ but just refer to the expression as $\A_{\vec c^{\j}(B)}$.
\end{bez}
%
\begin{ex}\label{b:c}
There are simple partitions of a multivector into commutative and anticommutative parts, like for example for $\A=a_0\e_0+a_1\e_1+a_2\e_2+a_{12}\e_{12}\in\G^{2,0}$ we get
\begin{equation}\begin{aligned}
\A_{\vec c^0(\e_1)}=&\frac 12(\A+\e_1^{-1}\A\e_1)\\
=&\frac 12(\A    +a_0+a_1\e_1-a_2\e_2-a_{12}\e_{12})\\
=&a_0+a_1\e_1
\end{aligned}\end{equation}
and therefore $\A=\A_{\vec c^0(\e_1)}+\A_{\vec c^1(\e_1)}=a_0\e_0+a_1\e_1+a_2\e_2+a_{12}\e_{12}$.
But a  decompositions can not always be achieved by just splitting up the multivector into its blades with respect to a given basis. Sometimes the expressions of these parts are even longer than the multivector itself, for example $\A=\e_1$ satisfies
\begin{equation}\begin{aligned}
(\e_1)_{\vec c^0(\e_1+\e_2)}=&\frac 12(\e_1+(\e_1+\e_2)^{-1}\e_1(\e_1+\e_2))\\
=&\frac 12(\e_1+\frac12(\e_1+\e_2)\e_1(\e_1+\e_2))\\
=&\frac 12(\e_1+\frac12(\e_1+\e_2-\e_1\e_2\e_1-\e_1\e_2\e_2))\\
=&\frac 12(\e_1+\e_2)
\end{aligned}\end{equation}
and gets decomposed into $\e_1=(\e_1)_{\vec c^0(\e_1+\e_2)}+(\e_1)_{\vec c^1(\e_1+\e_2)}=\frac 12(\e_1+\e_2)+\frac 12(\e_1-\e_2)$.
\end{ex}
We will show that the decomposition of a multivector into commutative and anticommutative parts with respect to basis blades always is a decomposition into its blades along this basis. First consider one basis blade $\e_{\k}$, here $\vec c^0(\e_{\k})$ can be interpreted as a mapping $\vec c^0:\G^{p,q}\to\P(\{\j\subset\{1,...,n\},1\leq j_1,<...,<j_{\iota}\leq n\})$ of the multivector argument into the power set of all multi-indices $\j$ as in (\ref{mimv}) which indicate the basis blades of $\G^{p,q}$. The mapping $\vec c^0$ returns the blades of any multivector that commute with its argument $\e_{\k}$ and its counterpart $\vec c^1:\G^{p,q}\to\P(\{\j\subset\{1,...,n\},1\leq j_1,<...,<j_{\iota}\leq n\})$ returns the blades that anticommute. The next Lemma will justify this interpretation, but for better understanding we start with a motivational example.
\begin{ex}
 In the previous example 
the value of $\vec c^0(\e_1)$ would be $\{\{0\},\{1\}\}$ and $\vec c^1(\e_1)=\{\{2\},\{12\}\}$, so we could write
\begin{equation}\begin{aligned}
\A_{\vec c^0(\e_1)}=&\sum_{\j\in c^0(\e_1)}a_{\j}\e_{\j}=\sum_{\j\in \{\{0\},\{1\}\}}a_{\j}\e_{\j}=a_0\e_0+a_1\e_1,\\
\A_{\vec c^1(\e_1)}=&\sum_{\j\in c^1(\e_1)}a_{\j}\e_{\j}=\sum_{\j\in\{\{2\},\{12\}\}}a_{\j}\e_{\j}=a_2\e_2+a_{12}\e_{12}.
\end{aligned}\end{equation}
\end{ex}
\begin{lem}\label{l:decomp_along_B}
We denote the length of the multi-indices $\j,\k$ by $\iota$ and $\kappa$. For a basis blade $\e_{\k}$ and an arbitrary element $\A=\sum_{\j}a_{\j}\e_{\j}$ of $\G^{p,q}$ the multivectors $\A_{\vec c^0(\e_{\k})}$ and $\A_{\vec c^1(\e_{\k})}$ are a decomposition of $\A$ along the basis blades, that means
\begin{equation}\begin{aligned}
 \A_{\vec c^0(\e_{\k})}=&\sum_{\j\in\vec c^0(\e_{\k})}a_{\j}\e_{\j},\\
 \A_{\vec c^1(\e_{\k})}=&\sum_{\j\in\vec c^1(\e_{\k})}a_{\j}\e_{\j}
\end{aligned}\end{equation}
with $\vec c^0(\e_{\k})\cup\vec c^1(\e_{\k})=\{\j\subset\{1,...,n\},1\leq j_1,<...,<j_{\iota}\leq n\}$ as in (\ref{mimv}) and $\vec c^0(\e_{\k})\cap\vec c^1(\e_{\k})=\emptyset$ and the index sets $\vec c^0(\e_{\k}),\vec c^1(\e_{\k})$ take the forms 
\begin{equation}\begin{aligned}
\vec c^0(\e_{\k})=&\{\j\subset\{1,...,n\},1\leq j_1,<...,<j_{\iota}\leq n,\iota\kappa-\beta(\e_{\j},\e_{\k})\text{ even}\},\\
\vec c^1(\e_{\k})=&\{\j\subset\{1,...,n\},1\leq j_1,<...,<j_{\iota}\leq n,\iota\kappa-\beta(\e_{\j},\e_{\k})\text{ odd}\}.
\end{aligned}\end{equation}
\end{lem}
\begin{proof}
For $\A_{\vec c^0}(\e_{\k})$ we get
\begin{equation}\begin{aligned}
\A_{\vec c^0}(\e_{\k})=&\frac12(\A+\e_{\k}^{-1}\A\e_{\k})\\
=&\frac12\sum_{\j}(a_{\j}\e_{\j}+\e_{\k}^{-1}a_{\j}\e_{\j}\e_{\k})\\
\overset{(\ref{acom})}=&\frac12\sum_{\j}(a_{\j}\e_{\j}+(-1)^{\iota\kappa-\beta(\e_{\j},\e_{\k})}\e_{\k}^{-1}\e_{\k}a_{\j}\e_{\j})\\
=&\frac12\sum_{\j}a_{\j}\e_{\j}(1+(-1)^{\iota\kappa-\beta(\e_{\j},\e_{\k})})\\
=&\sum_{\j\in\vec c^0(\e_{\k})}a_{\j}\e_{\j}.
\end{aligned}\end{equation}
The proof for $\A_{\vec c^1}(\e_{\k})$ works analogously.
\end{proof}
\begin{rem}
An alternative way of describing the decomposition would be $\A_{\vec c^0(\e_{\k})}=\sum_{\j}a_{0\j}\e_{\j}$ and $\A_{\vec c^1(\e_{\k})}=\sum_{\j}a_{1\j}\e_{\j}$ with 
\begin{equation}\begin{aligned}
a_{0\j}=&\begin{cases}a_{\j}&\text{ for }\iota\kappa-\beta(\j,\k)\text{ even},\\
           0&\text{ else. }\end{cases}\\
\end{aligned}\end{equation}
The proof works analogously for $\A_{\vec c^1}(\e_{\k})=\sum_{\j}a_{1\j}\e_{\j}$ with
\begin{equation}\begin{aligned}
a_{1\j}=&\begin{cases}a_{\j}&\text{ for }\iota\kappa-\beta(\j,\k)\text{ odd},\\
           0&\text{ else. }\end{cases}
\end{aligned}\end{equation}
and $\A= \sum_{\j}(a_{0\j}+a_{1\j})\e_{\j}$ with $\forall\j:(a_{0\j}=a_{\j},a_{1\j}=0$) or $(a_{0\j}=0,a_{1\j}=a_{\j})$.
\end{rem}
%
%
\begin{lem}\label{l:decomp_along_B2}
 For basis blades $B=\{\e_{\k(1)},...,\e_{\k(d)}\}$ and multi-indices $\l\in\{0,1\}^d$ the $\A_{\vec c^{\l}(B)}$ form a decomposition of the multivector $\A$ along the basis blades, that means
\begin{equation}\begin{aligned}
 \A_{\vec c^{\l}(B)}=\sum_{\j\in\vec c^{\l}(B)}a_{\j}\e_{\j}
\end{aligned}\end{equation}
with $\bigcup_{l\in\{0,1\}^d}\vec c^{\l}(B)=\{\j\subset\{1,...,n\},1\leq j_1,<...,<j_{\iota}\leq n\}$ as in (\ref{mimv}) and $\forall \l\neq \l^\prime\in\{0,1\}^d:\vec c^{\l}(B)\cap\vec c^{\l^\prime}(B)=\emptyset$ and the index set $\vec c^{\l}(B)$ takes the form
\begin{equation}\label{c^j}
 \vec c^{\l}(\e_{\k(1)},...,\e_{\k(d)})=\bigcap_{\nu=1}^d \vec c^{l_{\nu}}(\e_{\k(\nu)}).
\end{equation}
\end{lem}
\begin{proof}
The assertion follows from multiple application of Lemma \ref{l:decomp_along_B} and the fact, that every part of $\A$ again is a multivector.
\begin{equation}\begin{aligned}
\A_{\vec c^{\l}(\B)}\overset{\text{def. }\ref{d:c^j}}=&((\A_{\vec c^{l_1}(\e_{\k(1)})})_{\vec c^{l_2}(\e_{\k(2)})}...)_{\vec c^{l_d}(\e_{\k(d)})},\\
\overset{\text{Lem. }\ref{l:decomp_along_B}}=&(\sum_{\j\in\vec c^{l_1}(\e_{\k(1)})}a_{\j}\e_{\j})_{\vec c^{l_2}(\e_{\k(2)})}...)_{\vec c^{l_d}(\e_{\k(d)})}\\
\overset{\text{Lem. }\ref{l:decomp_along_B}}=&(\sum_{\j\in\vec c^{l_1}(\e_{\k(1)})\text{ and }\j\in\vec c^{l_2}(\e_{\k(2)})}a_{\j}\e_{\j})_{\vec c^{l_3}(\e_{\k(3)})}...)_{\vec c^{l_d}(\e_{\k(d)})}\\
=&(\sum_{\j\in\vec c^{l_1}(\e_{\k(1)})\cap\vec c^{l_2}(\e_{\k(2)})}a_{\j}\e_{\j})_{\vec c^{l_3}(\e_{\k(3)})}...)_{\vec c^{l_d}(\e_{\k(d)})}\\
=&...\\
\overset{\text{Lem. }\ref{l:decomp_along_B}}=&\sum_{\j\in\bigcap_{\nu=1}^d \vec c^{l_{\nu}}(\e_{\k(\nu)}))}a_{\j}\e_{\j}\\
=&\sum_{\j\in\vec c^{\l}(B)}a_{\j}\e_{\j}
\end{aligned}\end{equation}
\end{proof}
\begin{rem}
Now for $d\in\N$ basis blades we can use the mapping $\vec c^{\l}:(\G^{p,q})^d\to\P(\{\j\subset\{1,...,n\},1\leq j_1,<...,<j_{\iota}\leq n\})$ to express the decomposition of a multivector. Compared to Definition \ref{d:c^j} it can be computed much faster using the formula (\ref{c^j}), which by the way again shows very clearly that the partition does not depend on the order of the blades.
\end{rem}
\begin{ex}
 Like in the two preceding examples 
we look at $\A\in\G^{2,0}$ but this time we use (\ref{c^j}). From
\begin{equation}
{\vec c^{0,0}({\e_{1},\e_{2}})}=\vec c^{0}(\e_{1})\cap\vec c^{0}(\e_{2})=\{\{0\},\{1\}\}\cap\{\{0\},\{2\}\}=\{0\}
\end{equation}
follows
\begin{equation}
\A_{\vec c^{0,0}({\e_{1},\e_{2}})}=\sum_{j\in\vec c^{0,0}({\e_{1},\e_{2}})}a_{�\j}\e_{\j}=a_0
\end{equation}
and computing the other parts analogously we get
\begin{equation}\begin{aligned}
\A=&\A_{\vec c^{0,0}({\e_{1},\e_{2}})}+\A_{\vec c^{1,0}({\e_{1},\e_{2}})}+\A_{\vec c^{0,1}({\e_{1},\e_{2}})}+\A_{\vec c^{1,1}({\e_{1},\e_{2}})}\\
=&a_0+a_1\e_1+a_2\e_2+a_{12}\e_{12}\\
=&\sum_{j}a_{\j}\e_{\j}.
\end{aligned}\end{equation}
\end{ex}
\begin{rem}
The decomposition with respect to basis blades is independent from the multivector $\A$ and the total amount of parts that occur is limited by $\min\{2^d,2^n\}$, where $d$ is the number of blades in $B$ and $n=p+q$ the dimension of the underlying vector space. In the case of the previous example this means that a higher $d$ would not result in a finer segmentation of $\A$. Certain combinations of commutation properties will just remain empty, for instance
\begin{equation}
\vec c^{1,1,1}(\e_{1},\e_{2},\e_{12})=
\emptyset.
\end{equation}
\end{rem}
%
%
Now we take a look at the decomposition of exponentials of functions as they appear in (\ref{gft}) that satisfy condition (\ref{bed}) and show that they take a very simple form.
%
\begin{lem}\label{l:2parts}
Let the value of $f(\x,\u):\R^m\times \R^m\to\mathscr I^{p,q}$ be a real multiple of a basis blade $\forall \x,\u\in\R^m$, $f(\x,\u)=\sgn(f(\x,\u))|f(\x,\u)|\e_{\k(\x,\u)}$ like in (\ref{bed}). The decompositions $e^{-f(\x,\u)}_{\vec c^{0}(\e_{\l})},e^{-f(\x,\u)}_{\vec c^{1}(\e_{\l})}$ of the exponential with respect to any basis blade $\e_{\l}\in\G^{p,q}$ can only take two different shapes:
\begin{equation}\begin{aligned}
e^{-f(\x,\u)}_{\vec c^{0}(\e_{\l})}=&\begin{cases}e^{-f(\x,\u)}&\text{ if }\k(\x,\u)\in\vec c^{0}(\e_{\l}),\\ 
\cos(|f(\x,\u)|)&\text{ if }\k(\x,\u)\notin\vec c^{0}(\e_{\l}),\end{cases}\\
e^{-f(\x,\u)}_{\vec c^{1}(\e_{\l})}=&\begin{cases} -\frac{f(\x,\u)}{|f(\x,\u)|}\sin(|f(\x,\u)|)&\text{ if }\k(\x,\u)\in\vec c^{1}(\e_{\l}),\\
0&\text{ if }\k(\x,\u)\notin\vec c^{1}(\e_{\l}).\end{cases}\\
\end{aligned}\end{equation}
\end{lem}
\begin{proof}
The exponential can be expressed as its sine and its cosine part.
So the decomposition into basis blades
\begin{equation}\begin{aligned}\label{sincosin}
e^{-f(\x,\u)}=&\sum_{\j}a_{\j}\e_{\j}\\
=&\cos(|f_l(\x,\u)|)\e_{0}-\sgn(f(\x,\u))\sin(|f_l(\x,\u)|)\e_{\k(\x,\u)}
\end{aligned}\end{equation}
has only two coefficients $a_{0}$ and $a_{\k(\x,\u)}$ different from zero. Lemma \ref{l:decomp_along_B} showed that the blades $\e_{\j}$ of (\ref{sincosin}) are sorted either into $\vec c^0$ or into $\vec c^1$ during the decomposition. Because $\e_0$ is real, it commutes with anything and therefore the cosine always belongs to $\vec c^0(\e_{\l})$ and never to $\vec c^1(\e_{\l})$. For the appearance of the sine we have to distinguish whether or not $\k(\x,\u)\in\vec c^0(\e_{\l})$ or $\k(\x,\u)\in\vec c^1(\e_{\l})$, which leads to the assertion.
\end{proof}

\begin{lem}\label{l:4parts}
Let $f(\x,\u):\R^m\times \R^m\to\mathscr I^{p,q}$ satisfy property (\ref{bed}) $\forall \x,\u\in\R^m$, $\l\in\{0,1\}^d$ a multi-index and $B=\{\e_{\k(1)},...,\e_{\k(d)}\}$ be a set of basis blades. The decompositions of the exponential with respect to $B$ can only take four different shapes:
\begin{equation}\begin{aligned}\label{4shapes}
e^{-f(\x,\u)}_{\vec c^{\l}(B)}=&\begin{cases}e^{-f(\x,\u)}&\text{ if }\l=0\text{ and }{\k(\x,\u)}\in\vec c^{\l}(B), \\
\cos(|f(\x,\u)|)&\text{ if }\l=0\text{ and }{\k(\x,\u)}\notin\vec c^{\l}(B),\\
-\frac{f(\x,\u)}{|f(\x,\u)|}\sin(|f_l(\x,\u)|)&\text{ if }\l\neq0\text{ and }{\k(\x,\u)}\in\vec c^{\l}(B),\\
0&\text{ if }\l\neq0\text{ and }{\k(\x,\u)}\notin\vec c^{\l}(B).\end{cases}\\
\end{aligned}\end{equation}
\end{lem}
\begin{proof}
 In Lemma \ref{l:2parts} we saw that for a fixed $\lambda\in\{0,...,d\}$ the ${\vec c^{1}(\e_{\k(\lambda)})}$ always removes the cosine part, so it will only remain for $l_1=...=l_d=0$. Analogously that the sine part is removed when $\k(\x,\u)\notin\vec c^{l_\lambda}(\e_{\k(\lambda)})$.
Because of $\vec c^{\l}(\e_{\k(1)},...,\e_{\k(d)})=\bigcap_{\lambda=1}^d \vec c^{l_{\lambda}}(\e_{\k(\lambda)})$, compare Lemma \ref{l:decomp_along_B2}, one appearance of such a case is sufficient to eliminate each part from the whole exponential.
\end{proof}
\begin{rem}
The shift theorem introduced in \cite{BSH11} takes a simpler form for GFTs satisfying property (\ref{bed}), that means for all transforms in the first example
. The simplification results from the predictable shape of the decomposition of exponentials with respect to basis blades from Lemma \ref{l:4parts}.
\end{rem}
%
\section{Geometric Convolution Theorem}
%
%
We have seen that coorthogonal blade functions can be expressed as real multiples of basis blades $f(\x,\u)= \sgn(f(\x,\u))|f(\x,\u)|\e_{\j(\x,\u)}:\R^m\times\R^m\to\mathscr I^{p,q}$ in Notation \ref{b:coorth}.
Lemma \ref{l:decomp_along_B2} guarantees for basis blade functions, that during the decomposition into commutative and anticommutative parts of a multivector no additional terms appear in the sum over the basis blades (\ref{mimv}). Each part is a real fragment of the multivector along the basis blades of the orthogonal basis from Theorem \ref{t:orth_basis}. Because of that an exponential can only become decomposed into four different shapes: itself, a cosine, a basis blade multiplied with a sine or zero, compare Lemma \ref{l:4parts}. This motivates the generalization of geometric Fourier transforms (\ref{gft}) to trigonometric transforms. We will just use it as an auxiliary construction here and analyze its properties and applications in a future paper.
\begin{defn}[Geometric Trigonometric Transform]\label{d:gtt}
 Let $\A:\R^m\to\G^{p,q}$ be a multivector field and $\x,\u\in\R^m$ vectors, $F_1,F_2$
two ordered finite sets of $\mu$, respectively $\nu-\mu$, mappings $\R^m\times \R^m\to\mathscr I^{p,q}$, $G_1,G_2$ two ordered finite sets of $\mu$, respectively $\nu-\mu$, mappings $(\R^m\times \R^m\to\mathscr I^{p,q})\to\G^{p,q}$ with each $g_l(-f_l(\x,\u))\forall l=1,...,\nu$ having one of the shapes from (\ref{4shapes}):
\begin{equation}\begin{aligned}\label{g_in_gtt}
g_l(-f_l(\x,\u)) =&\begin{cases}e^{-f_l(\x,\u)}\\
\cos(|f_l(\x,\u)|)\\
-\frac{f_l(\x,\u)}{|f_l(\x,\u)|}\sin(|f_l(\x,\u)|),\\
0.\end{cases}
\end{aligned}\end{equation}
The \textbf{Geometric Trigonometric Transform} (GTT) $\F_{G_1(F_1),G_2(F_2)}(\A)$ is defined by 
\begin{equation}
 \F_{G_1(F_1),G_2(F_2)}(\A)(\u):=\int_{\R^m}\prod_{l=1}^{\mu}g_l({-f_l(\x,\u)})\A(\x)\prod_{l=\mu+1}^{\nu}g_l({-f_l(\x,\u)}).
\end{equation}
\end{defn}
\begin{bez}
We have seen in Lemma \ref{l:4parts} that the decomposition of an exponential with respect to basis blades takes the same shape like the functions $G_1,G_2$ of a GTT (\ref{g_in_gtt}). Therefore for a geometric Fourier transform with basis blade functions $F_1,F_2$, two sets of basis blades $B_1=\{\e_{\k(1)},...,\e_{\k(\eta)}\}$, $B_2=\{\e_{\k(\eta+1)},...,\e_{\k(\theta-\eta)}\}$ and strictly lower and upper triangular matrices\footnote{These matrices were introduced originally in Lemma 6.8 in \cite{BSH11}. It is repeated in this work as Lemma \ref{l:expcom2}. The proof can be found in \cite{BSH11}.} $J\in\{0,1\}^{\mu\times\eta}$, $K\in\{0,1\}^{(\nu-\mu)\times\theta}$ whose rows are $\mu$ and $\nu-\mu$ multi-indices $(J)_l\in\{0,1\}^{\eta}$ respectively $(K)_l\in\{0,1\}^{\theta}$, we can construct a geometric trigonometric transform $\F_{G_1(F_1),G_2(F_2)}(\A)$ by setting $g_l(-f_l(\x,\u))=e^{-f_l(\x,\u)}_{\vec c^{(J)_l}}$ for $l=1,...,\mu$ and $g_l(-f_l(\x,\u))=e^{-f_l(\x,\u)}_{\vec c^{(K)_l}}$ for $l=\mu+1,.
..,\nu$. We refer to it
shortly as
\begin{equation}
 \F_{(F_1)_{\vec c^{J}(B_1)},(F_2)_{\vec c^{K}(B_2)}}(\A)(\u):=\int_{\R^m}\prod_{l=1}^{\mu}e^{-f_l(\x,\u)}_{\vec c^{(J)_l}(B_1)}\A(\x)\prod_{l=\mu+1}^{\nu}e^{-f_l(\x,\u)}_{\vec c^{(K)_{l-\mu}}(B_2)}\d^m \x.
\end{equation}
In the case of $\F_{(F_1)_{\vec c^{J}(F_1)},(F_2)_{\vec c^{K}(F_2)}}$ we will only write $ \F_{(F_1)_{\vec c^{J}},(F_2)_{\vec c^{K}}}$.
\end{bez}
The geometric trigonometric transform is a generalization of the geometric Fourier transform from (\ref{gft}). We will use it to prove the convolution theorem of the GFT. To accomplish this we additionally need the following facts shown in \cite{BSH11}. Please note that for the proofs of all Lemmata from \cite{BSH11} the claim for the set of functions to be basis blades functions is not necessary.
\begin{defn}
We call a GFT \textbf{left (right) separable}, if
\begin{equation}
\label{ifix}
f_l=|f_l(\x,\u)|i_l(\u),
\end{equation}
$\forall l=1,...,\mu$, ($l=\mu+1,...,\nu$), where $|f_l(\x,\u)|:\R^m\times\R^m\to\R$ is a real function and $i_l:\R^m\to\mathscr I^{p,q}$ a function that does not depend on $\x$.
\end{defn}
\begin{lem}\label{l:expcom}
  Let $F=\{f_1(\x,\u),...,f_{d}(\x,\u)\}$ be a set of pointwise invertible functions then the ordered product of their exponentials and an arbitrary multivector $\A\in\G^{p,q}$ satisfies
\begin{equation}\begin{aligned}
 \prod_{l=1}^{d}e^{-f_l(\x,\u)}\A=\sum_{\j\in\{0,1\}^d}\vec A_{\vec c^{\j}(\overleftarrow F)}(\x,\u)\prod_{l=1}^{d}e^{-(-1)^{j_l}f_l(\x,\u)},
\end{aligned}\end{equation}
where $A_{\vec c^{\j}(\overleftarrow F)}(\x,\u):=A_{\vec c^{\j}(\overleftarrow {F(\x,\u)})}$ is a multivector valued function $\R^m\times\R^m\to\G^{p,q}$.
\end{lem}
\begin{lem}\label{l:prodtrennen}
 Let $F=\{f_1(\x,\u),...,f_{d}(\x,\u)\}$ be a set of separable functions that are linear with respect to $\x$. Further let $J\in\{0,1\}^{d\times d}$ be a strictly lower triangular matrix, that is associated column by column with a multi-index ${\j\in\{0,1\}^{d}}$ by $\forall k=1,...,d:
(\sum_{l=1}^{d}J_{l,k})\bmod 2=j_{k}$, with $(J)_l$ being its $l$-th row, then 
\begin{equation}\begin{aligned}\label{prodtrennen1}
&\prod_{l=1}^d e^{-f_l(\x+\y,\u)}
\\&= \sum_{\j\in\{0,1\}^{d}}\sum_{\stackrel{J\in\{0,1\}^{d\times d},}{\sum_{l=1}^d(J)_l\bmod 2=\j}}\prod_{l=1}^{d} e^{-f_l(\x,\u)}_{\vec c^{(J)_l}(\overleftarrow{f_1,...,f_l,0,...,0})}\prod_{l=1}^{d}e^{-(-1)^{j_l}f_l(\y,\u)}
\end{aligned}\end{equation}
or alternatively with strictly upper triangular matrices $J$
\begin{equation}\begin{aligned}\label{prodtrennen2}
&\prod_{l=1}^d e^{-f_l(\x+\y,\u)}
\\&=\sum_{\j\in\{0,1\}^{d}}\sum_{\stackrel{J\in\{0,1\}^{d\times d},}{\sum_{l=1}^d(J)_l\bmod 2=\j}}\prod_{l=1}^{d}e^{-(-1)^{j_l}f_l(\x,\u)}\prod_{l=1}^{d} e^{-f_l(\y,\u)}_{\vec c^{(J)_l}(\overrightarrow{0,...,0,f_l,...,f_d})}.
\end{aligned}\end{equation}
\end{lem}
\begin{defn}\label{d:fj}
 For a set of functions $F=\{f_1(\x,\u),...,f_{d}(\x,\u)\}$ and a multi-index ${\j\in\{0,1\}^d}$, we define the set of functions $F(\j)$ by
\begin{equation}\begin{aligned}
 F(\j):=\{(-1)^{j_1}f_1(\x,\u),...,(-1)^{j_{d}}f_{d}(\x,\u)\}.
\end{aligned}\end{equation}
\end{defn}
We also need a generalization of Lemma \ref{l:expcom} that allows us to swap the order of partial exponentials and multivectors.
\begin{lem}\label{l:expcom2}
For sets of functions $F=\{f_1(\x,\u),...,f_{d}(\x,\u)\},G=\{g_1,...,g_{d}\}$ like in (\ref{g_in_gtt}) we get analogously to Lemma \ref{l:expcom}
\begin{equation}
 \prod_{l=1}^{d}g_l({-f_l(\x,\u)})\A=\sum_{\j\in\{0,1\}^d}\vec A_{\vec c^{\j}(F)}\prod_{l=1}^{d}g_l({-(-1)^{j_l}f_l(\x,\u)}).
\end{equation}
\end{lem}
\begin{proof}
 First we analyze the interaction of $\A$ with one partial exponential $g_l({-f_l(\x,\u)})$. It can take three different shapes 
\begin{equation}\begin{aligned}
g_l(-f_l(\x,\u)) =&\begin{cases}e^{-f_l(\x,\u)}\\
\cos(|f_l(\x,\u)|)\\
-\frac{f_l(\x,\u)}{|f_l(\x,\u)|}\sin(|f_l(\x,\u)|),\\
0.\end{cases}
\end{aligned}\end{equation}
In the first case lemma \ref{l:expcom} proves the assertion and the last one is trivial.
Assume the second case and note that then $g_l(-f_l(\x,\u))$ equals $g_l(f_l(\x,\u))$ because of the symmetry of the cosine.
\begin{equation}\begin{aligned}
 g_l(-f_l(\x,\u))\A=&\cos(|f_l(\x,\u)|)\A\\
\overset{\text{cos. }\in\R}=&\A\cos(|f_l(\x,\u)|)\\
\overset{\text{Lem. }\ref{l:prodviele}}=&\A_{\vec c^{0}(f_l)}\cos(|f_l(\x,\u)|)+\A_{\vec c^{1}(f_l)}\cos(|f_l(\x,\u)|)\\
=&\A_{\vec c^{0}(f_l)}g_l(-f_l(\x,\u))+\A_{\vec c^{1}(f_l)}g_l(f_l(\x,\u))
\end{aligned}\end{equation}
In the third case we have
\begin{equation}\begin{aligned}
 g_l(-f_l(\x,\u))\A=&-\frac{f_l(\x,\u)}{|f_l(\x,\u)|}\sin(|f_l(\x,\u)|)\A\\
\overset{\text{Lem. }\ref{l:prodviele}}=&\A_{\vec c^{0}(f_l)}\frac{-f_l(\x,\u)}{|f_l(\x,\u)|}\sin(|f_l(\x,\u)|)
\\&+\A_{\vec c^{1}(f_l)}\frac{f_l(\x,\u)}{|f_l(\x,\u)|}\sin(|f_l(\x,\u)|)\\
=&\A_{\vec c^{0}(f_l)}g_l(-f_l(\x,\u))+\A_{\vec c^{1}(f_l)}g_l(f_l(\x,\u)).
\end{aligned}\end{equation}
So in all cases we have
\begin{equation}\begin{aligned}
 g_l(-f_l(\x,\u))\A=&\A_{\vec c^{0}(f_l)}g_l(-f_l(\x,\u))+\A_{\vec c^{1}(f_l)}g_l(f_l(\x,\u))\\
=&\sum_{\j\in\{0,1\}^1}\A_{\vec c^{\j}(f_l)}\prod_{l=1}^{1}g_l({-(-1)^{j_l}f_l(\x,\u)}).
\end{aligned}\end{equation}
Applying it repeatedly to the whole product like in the proof of Lemma \ref{l:expcom} in \cite{BSH11} leads to the assertion.
\end{proof}
\begin{defn}
 Let $\A(\x),\B(\x):\R^m\to\G^{p,q}$ be two multivector fields. Their \textbf{convolution} $(\vec{A}*\B)(\x)$ is defined as
\begin{equation}\begin{aligned}
(\vec{A}*\B)(\x):=&\int_{\R^m}\vec{A}(\y)\B(\x-\y)\d^m\y.
\end{aligned}\end{equation}
\end{defn}

\begin{thm}[convolution]\label{t:convolution}
Let $\A,\B,\vec C:\R^m\to\G^{p,q}$ be multivector fields with $\A(\x)=(\vec{C}*\B)(\x)$ and $F_1,F_2$ be coorthogonal, separable and linear with respect to the first argument, ${\j,\j^\prime\in\{0,1\}^{\mu}},\k,\k^\prime\in\{0,1\}^{(\nu-\mu)}$ and $J\in\{0,1\}^{\mu\times \mu}$ and $K\in\{0,1\}^{(\nu-\mu) \times (\nu-\mu)}$ are the strictly lower, respectively upper, triangular matrices with rows $(J)_{l},(K)_{l-\mu}$ summing up to $(\sum_{l=1}^{\mu}(J)_{l})\bmod 2=\j$ respectively $(\sum_{l=\mu+1}^{\nu}(K)_{l-\mu})\bmod 2=\k$ as in Lemma \ref{l:prodtrennen}, then the geometric Fourier transform of $\A$ satisfies the convolution property
\begin{equation}\begin{aligned}
&\F_{F_1,F_2}(\A)(\u)
\\&=\sum_{\j,\j^\prime,\k,\k^\prime}\sum_{J,K}
(\F_{F_1(\j),F_2(\k+\k^\prime)}(\vec{C})(\u))_{\vec c^{\j^\prime}(F_1)}
\F_{(F_1(\j^\prime))_{\vec c^{J}},(F_2)_{\vec c^{K}}}(\B_{\vec c^{\k^\prime}(F_2)})(\u).
\end{aligned}\end{equation}
\end{thm}
\begin{proof}
 \begin{equation}\begin{aligned}
 &\F_{F_1,F_2}(\A)(\u)=\int_{\R^m}\prod_{f\in F_1}e^{-f(\x,\u)}(\vec{C}*\B)(\x)\prod_{f\in F_2}e^{-f(\x,\u)}\d^m \x\\
&=\int_{\R^m}\prod_{f\in F_1}e^{-f(\x,\u)}\int_{\R^m}\vec{C}(\y)\B(\x-\y)\d^m\y \prod_{f\in F_2}e^{-f(\x,\u)}\d^m \x\\
&\overset{\x-\y=\z}=\int_{\R^m}\int_{\R^m}\prod_{f\in F_1}e^{-f(\z+\y,\u)}\vec{C}(\y)\B(\z) \prod_{f\in F_2}e^{-f(\z+\y,\u)}\d^m\y\d^m \z\\
\end{aligned}\end{equation}
We separate the products into parts that only depend on $\y$ and ones that only depend on $\z$.
 \begin{equation}\begin{aligned}
\overset{\text{Lem. }\ref{l:prodtrennen}}=&
\int_{\R^m}\int_{\R^m}
\sum_{\j\in\{0,1\}^{\mu}}\sum_{\stackrel{J\in\{0,1\}^{\mu\times \mu}}{\sum(J)_l\bmod 2=\j}}\prod_{l=1}^{\mu} e^{-f_l(\z,\u)}_{\vec c^{(J)_l}(F_1)}
\\&\prod_{l=1}^{\mu}e^{-(-1)^{j_l}f_l(\y,\u)}\vec{C}(\y)\B(\z) 
\sum_{\k\in\{0,1\}^{\nu-\mu}}\sum_{\stackrel{K\in\{0,1\}^{\nu-\mu\times \nu-\mu}}{\sum(K)_l\bmod 2=\j}}
\\&\prod_{l=\mu+1}^{\nu}e^{-(-1)^{k_{l-\mu}}f_l(\y,\u)}\prod_{l=\mu+1}^{\nu}e^{-f_l(\z,\u)}_{\vec c^{(K)_{l-\mu}}(F_2)}
\d^m\y\d^m \z\\
=&\sum_{\j,\k}\sum_{J,K}\int_{\R^m}\int_{\R^m}
\prod_{l=1}^{\mu} e^{-f_l(\z,\u)}_{\vec c^{(J)_l}(F_1)}\prod_{l=1}^{\mu}e^{-(-1)^{j_l}f_l(\y,\u)}
\vec{C}(\y) \\&
\B(\z)
\prod_{l=\mu+1}^{\nu}e^{-(-1)^{k_{l-\mu}}f_l(\y,\u)}\prod_{l=\mu+1}^{\nu}e^{-f_l(\z,\u)}_{\vec c^{(K)_{l-\mu}}(F_2)}
\d^m\y\d^m \z\\
 \end{aligned}\end{equation}
Next step is to collect all parts that depend on $\y$.
 \begin{equation}\begin{aligned}
\overset{\text{Lem. }\ref{l:expcom}}=&\sum_{\j,\k}\sum_{J,K}\int_{\R^m}\int_{\R^m}
\prod_{l=1}^{\mu} e^{-f_l(\z,\u)}_{\vec c^{(J)_l}(F_1)}\prod_{l=1}^{\mu}e^{-(-1)^{j_l}f_l(\y,\u)}
\vec{C}(\y)
\\&
\sum_{\k^\prime\in\{0,1\}^{\nu-\mu}}
\prod_{l=\mu+1}^{\nu}e^{-(-1)^{k_{l-\mu}+k^\prime_{l-\mu}}f_l(\y,\u)}\B_{\vec c^{\k^\prime}(F_2)}(\z)
\\&\prod_{l=\mu+1}^{\nu}e^{-f_l(\z,\u)}_{\vec c^{(K)_{l-\mu}}(F_2)}
\d^m\y\d^m \z\\
=&\sum_{\j,\k,\k^\prime}\sum_{J,K}\int_{\R^m}
\prod_{l=1}^{\mu} e^{-f_l(\z,\u)}_{\vec c^{(J)_l}(F_1)}\int_{\R^m}\prod_{l=1}^{\mu}e^{-(-1)^{j_l}f_l(\y,\u)}
\vec{C}(\y)\\&
\prod_{l=\mu+1}^{\nu}e^{-(-1)^{k_{l-\mu}+k^\prime_{l-\mu}}f_l(\y,\u)}\d^m\y
\B_{\vec c^{\k^\prime}(F_2)}(\z)
\\& \prod_{l=\mu+1}^{\nu}e^{-f_l(\z,\u)}_{\vec c^{(K)_{l-\mu}}(F_2)}
\d^m \z\\
=&\sum_{\j,\k,\k^\prime}\sum_{J,K}\int_{\R^m}
\prod_{l=1}^{\mu} e^{-f_l(\z,\u)}_{\vec c^{(J)_l}(F_1)}
\F_{F_1(\j),F_2(\k+\k^\prime)}(\vec{C})(\u)
\B_{\vec c^{\k^\prime}(F_2)}(\z)
\\&\prod_{l=\mu+1}^{\nu}e^{-f_l(\z,\u)}_{\vec c^{(K)_{l-\mu}}(F_2)}
\d^m \z\\
 \end{aligned}\end{equation}
Finally we gather the parts depending on $\z$.
 \begin{equation}\begin{aligned}
\overset{\text{Lem. }\ref{l:expcom2}}=&\sum_{\j,\k,\k^\prime}\sum_{J,K}\int_{\R^m}
\sum_{\j^\prime\in\{0,1\}^{\mu}}(\F_{F_1(\j),F_2(\k+\k^\prime)}(\vec{C})(\u))_{\vec c^{\j^\prime}(F_1)}
\\&\prod_{l=1}^{\mu} e^{-(-1)^{j^\prime_l}f_l(\z,\u)}_{\vec c^{(J)_l}(F_1)}
\B_{\vec c^{\k^\prime}(F_2)}(\z)
\prod_{l=\mu+1}^{\nu}e^{-f_l(\z,\u)}_{\vec c^{(K)_{l-\mu}}(F_2)}
\d^m \z\\
=&\sum_{\j,\j^\prime,\k,\k^\prime}\sum_{J,K}
(\F_{F_1(\j),F_2(\k+\k^\prime)}(\vec{C})(\u))_{\vec c^{\j^\prime}(F_1)}\\&
\int_{\R^m}\prod_{l=1}^{\mu} e^{-(-1)^{j^\prime_l}f_l(\z,\u)}_{\vec c^{(J)_l}(F_1)}
\B_{\vec c^{\k^\prime}(F_2)}(\z)
\prod_{l=\mu+1}^{\nu}e^{-f_l(\z,\u)}_{\vec c^{(K)_{l-\mu}}(F_2)}
\d^m \z\\
=&\sum_{\j,\j^\prime,\k,\k^\prime}\sum_{J,K}
(\F_{F_1(\j),F_2(\k+\k^\prime)}(\vec{C})(\u))_{\vec c^{\j^\prime}(F_1)}
\\&\F_{(F_1(\j^\prime))_{\vec c^{J}},(F_2)_{\vec c^{K}}}(\B_{\vec c^{\k^\prime}(F_2)})(\u)
\end{aligned}\end{equation}
\end{proof}
\begin{rem}
The formula in the convolution theorem can take various other shapes depending on the way Lemma \ref{l:prodtrennen} is applied. In Theorem \ref{t:convolution} we used Lemma \ref{l:prodtrennen} in its first version (\ref{prodtrennen1}) on $F_1$ and in its second version (\ref{prodtrennen2}) on $F_2$. This has the advantage that the GTT is needed only on one side. We get the same effect with its second version (\ref{prodtrennen2}) on $F_1$ and its first version (\ref{prodtrennen2}) on $F_2$ by
 \begin{equation}\begin{aligned}
&\F_{F_1,F_2}(\A)(\u)
\\&=\sum_{\j,\j^\prime,\k,\k^\prime}\sum_{J,K}
(\F_{(F_1)_{\vec c^{J}},(F_2(\k^\prime))_{\vec c^{K}}}(\vec{C})(\u))_{\vec c^{\j^\prime}(F_1)}
\F_{F_1(\j+\j^\prime),F_2(\k)}(\B_{\vec c^{\k^\prime}(F_2)})(\u).
\end{aligned}\end{equation}
Using the first version twice leads to
 \begin{equation}\begin{aligned}
&\F_{F_1,F_2}(\A)(\u)
\\&=\sum_{\j,\j^\prime,\k,\k^\prime}\sum_{J,K}
(\F_{F_1(\j),(F_2(\k^\prime))_{\vec c^K}}(\vec{C})(\u))_{\vec c^{\j^\prime}(F_1)}
\F_{(F_1(\j^\prime))_{\vec c^J},F_2(\k)}(\B_{\vec c^{\k^\prime}(F_2)})(\u),
\end{aligned}\end{equation}
and using the second twice to
 \begin{equation}\begin{aligned}
&\F_{F_1,F_2}(\A)(\u)
\\&=\sum_{\j,\j^\prime,\k,\k^\prime}\sum_{J,K}
(\F_{(F_1)_{\vec c^{J}},F_2(\k+\k^\prime)}(\vec{C})(\u))_{\vec c^{\j^\prime}(F_1)}
\F_{F_1(\j+\j^\prime),(F_2)_{\vec c^{K}}}(\B_{\vec c^{\k^\prime}(F_2)})(\u).
\end{aligned}\end{equation}
These versions have the advantage of being a bit more symmetric. During the proof of Theorem \ref{t:convolution} we started recomposing the transform around $\vec C$. Each of the four formulae has a counterpart that is constructed by restructuring around $\B$ first. Listed in the analog order they take the shapes
\begin{equation}\begin{aligned}
&\F_{F_1,F_2}(\A)(\u)
\\&=\sum_{\j,\j^\prime,\k,\k^\prime}\sum_{J,K}
\F_{F_1(\j),F_2(\k+\k^\prime)}(\vec{C}_{\vec c^{\j^\prime}(F_1)})(\u)
(\F_{(F_1(\j^\prime))_{\vec c^{J}},(F_2)_{\vec c^{K}}}(\B)(\u))_{\vec c^{\k^\prime}(F_2)},\\
&\F_{F_1,F_2}(\A)(\u)
\\&=\sum_{\j,\j^\prime,\k,\k^\prime}\sum_{J,K}
\F_{(F_1)_{\vec c^{J}},(F_2(\k^\prime))_{\vec c^{K}}}(\vec{C}_{\vec c^{\j^\prime}(F_1)})(\u)
(\F_{F_1(\j+\j^\prime),F_2(\k)}(\B)(\u))_{\vec c^{\k^\prime}(F_2)},\\
&\F_{F_1,F_2}(\A)(\u)
\\&=\sum_{\j,\j^\prime,\k,\k^\prime}\sum_{J,K}
\F_{F_1(\j),(F_2(\k^\prime))_{\vec c^K}}(\vec{C}_{\vec c^{\j^\prime}(F_1)})(\u)
(\F_{(F_1(\j^\prime))_{\vec c^J},F_2(\k)}(\B)(\u))_{\vec c^{\k^\prime}(F_2)},\\
&\F_{F_1,F_2}(\A)(\u)
\\&=\sum_{\j,\j^\prime,\k,\k^\prime}\sum_{J,K}
\F_{(F_1)_{\vec c^{J}},F_2(\k+\k^\prime)}(\vec{C}_{\vec c^{\j^\prime}(F_1)})(\u)
(\F_{F_1(\j+\j^\prime),(F_2)_{\vec c^{K}}}(\B)(\u))_{\vec c^{\k^\prime}(F_2)}.
\end{aligned}\end{equation}
Depending on the application one or some of these might be preferred compared to the others, because of savings in memory or runtime.
\end{rem}
\begin{cor}[convolution]\label{k:convolution}
 Let $\A,\B,\vec C:\R^m\to\G^{p,q}$ be multivector fields with $\A(\x)=(\vec{C}*\B)(\x)$ and $F_1,F_2$ each consist of mutually commutative functions\footnote{Cross commutativity is not necessary.}, being separable and linear with respect to the first argument and $\j^\prime\in\{0,1\}^{\mu},\k^\prime\in\{0,1\}^{(\nu-\mu)}$ multi-indices, then the geometric Fourier transforms satisfy the convolution property
\begin{equation}\begin{aligned}
 \F_{F_1,F_2}(\A)(\u)=&\sum_{\j^\prime,\k^\prime}(\F_{F_1,F_2(\k^\prime)}(\vec{C})(\u))_{\vec c^{\j^\prime}(F_1)}
\F_{F_1(\j^\prime)),F_2}(\B_{\vec c^{\k^\prime}(F_2)})(\u)
\end{aligned}\end{equation}
or
\begin{equation}\begin{aligned}
 \F_{F_1,F_2}(\A)(\u)=&\sum_{\j^\prime,\k^\prime}\F_{F_1,F_2(\k^\prime)}(\vec{C}_{\vec c^{\j^\prime}(F_1)})(\u),
(\F_{F_1(\j^\prime)),F_2}(\B)(\u))_{\vec c^{\k^\prime}(F_2)}.
\end{aligned}\end{equation}
If the values of the functions in $F_1$ and $F_2$ are in the center of $\G^{p,q}$ it even satisfies the simple product formula
\begin{equation}\begin{aligned}
 \F_{F_1,F_2}(\A)(\u)=&\F_{F_1,F_2}(\vec{C})(\u))\F_{F_1,F_2}(\B)(\u).
\end{aligned}\end{equation}
\end{cor}
\begin{ex}
We summarize the exact shape of the convolution of multivector fields under the transforms from the first Example 
using the same order.
\begin{enumerate}
\item The Clifford Fourier transform from \cite{Janc90,Ebl06,HM08} takes the form
\begin{equation}\begin{aligned}
\F_{f_1}(\A)=\F_{f_1}(\vec C_{\vec c^{0}(i)})\F_{f_1}(\B)+\F_{-f_1}(\vec C_{\vec c^{1}(i)})\F_{f_1}(\B)
\end{aligned}\end{equation}
for $n=2\pmod4$ and for $n=3\pmod4$ the even simpler one
\begin{equation}\begin{aligned}
\F_{f_1}(\A)=\F_{f_1}(\vec C)\F_{f_1}(\B)
\end{aligned}\end{equation}
because in this case the pseudoscalar is in the center of $\G^{n,0}$. 
\item The Sommen B\"ulow Clifford Fourier transform \cite{Som82,Bue99} is the only one of our examples that does not fulfill the constraints of Corollary \ref{k:convolution} but the ones of Theorem \ref{t:convolution}
\begin{equation}\begin{aligned}
\F_{f_1,...,f_n}(\A)=&\sum_{\k,\k^\prime\in\{0,1\}^{n}}\sum_{K}
\F_{(-1)^{k_1+k^\prime_1}f_1,...,(-1)^{k_n+k^\prime_n}f_n}(\vec{C})\\
&\F_{(f_1,...,f_n)_{\vec c^{K}}}(\B_{\vec c^{\k^\prime}(f_1,...,f_n)})
\end{aligned}\end{equation}
with strictly upper triangular matrices in $\{0,1\}^{n \times n}$ with rows $(K)_{l-\mu}$ summing up to $(\sum_{l=\mu+1}^{\nu}(K)_{l-\mu})\bmod 2=\k$.
\item The quaternionic Fourier transform \cite{Ell93,Bue99} has the shape
\begin{equation}\begin{aligned}
\F_{f_1,f_2}(\A)=& (\F_{f_1,f_2}(\vec{C}))_{\vec c^0(f_1)}\F_{f_1,f_2}(\vec B_{\vec c^0(f_2)})
\\&+(\F_{f_1,f_2}(\vec{C}))_{\vec c^1(f_1)}\F_{-f_1,f_2}(\vec B_{\vec c^0(f_2)}).
\end{aligned}\end{equation}
\item And the spacetime Fourier transform \cite{Hitz07} has exactly the same shape
\begin{equation}\begin{aligned}
\F_{f_1,f_2}(\A)=& (\F_{f_1,f_2}(\vec{C}))_{\vec c^0(f_1)}\F_{f_1,f_2}(\vec B_{\vec c^0(f_2)})
\\&+(\F_{f_1,f_2}(\vec{C}))_{\vec c^1(f_1)}\F_{-f_1,f_2}(\vec B_{\vec c^0(f_2)}).
\end{aligned}\end{equation}
\item The Clifford Fourier transform for color images \cite{BBS08} takes the rather long form
\begin{equation}\begin{aligned}
\F_{f_1,f_2,f_3,f_4}(\A)=&(\F_{f_1,f_2,f_3,f_4}(\vec{C}))_{\vec c^{00}(f_1,f_2)}\F_{f_1,f_2,f_3,f_4}(\vec B_{\vec c^{00}(f_3,f_4)}) 
\\ &+(\F_{f_1,f_2,f_3,f_4}(\vec{C}))_{\vec c^{01}(f_1,f_2)}\F_{f_1,-f_2,f_3,f_4}(\vec B_{\vec c^{00}(f_3,f_4)}) 
\\ & +(\F_{f_1,f_2,f_3,f_4}(\vec{C}))_{\vec c^{10}(f_1,f_2)}\F_{-f_1,f_2,f_3,f_4}(\vec B_{\vec c^{00}(f_3,f_4)}) 
\\ & +(\F_{f_1,f_2,f_3,f_4}(\vec{C}))_{\vec c^{11}(f_1,f_2)}\F_{-f_1,-f_2,f_3,f_4}(\vec B_{\vec c^{00}(f_3,f_4)}) 
\\ & +(\F_{f_1,f_2,f_3,-f_4}(\vec{C}))_{\vec c^{00}(f_1,f_2)}\F_{f_1,f_2,f_3,f_4}(\vec B_{\vec c^{01}(f_3,f_4)}) 
\\ & +(\F_{f_1,f_2,f_3,-f_4}(\vec{C}))_{\vec c^{01}(f_1,f_2)}\F_{f_1,-f_2,f_3,f_4}(\vec B_{\vec c^{01}(f_3,f_4)}) 
\\ & +(\F_{f_1,f_2,f_3,-f_4}(\vec{C}))_{\vec c^{10}(f_1,f_2)}\F_{-f_1,f_2,f_3,f_4}(\vec B_{\vec c^{01}(f_3,f_4)}) 
\\ & +(\F_{f_1,f_2,f_3,-f_4}(\vec{C}))_{\vec c^{11}(f_1,f_2)}\F_{-f_1,-f_2,f_3,f_4}(\vec B_{\vec c^{01}(f_3,f_4)}) 
\\ & +(\F_{f_1,f_2,-f_3,f_4}(\vec{C}))_{\vec c^{00}(f_1,f_2)}\F_{f_1,f_2,f_3,f_4}(\vec B_{\vec c^{10}(f_3,f_4)}) 
\\ & +(\F_{f_1,f_2,-f_3,f_4}(\vec{C}))_{\vec c^{01}(f_1,f_2)}\F_{f_1,-f_2,f_3,f_4}(\vec B_{\vec c^{10}(f_3,f_4)}) 
\\ & +(\F_{f_1,f_2,-f_3,f_4}(\vec{C}))_{\vec c^{10}(f_1,f_2)}\F_{-f_1,f_2,f_3,f_4}(\vec B_{\vec c^{10}(f_3,f_4)}) 
\\ & +(\F_{f_1,f_2,-f_3,f_4}(\vec{C}))_{\vec c^{11}(f_1,f_2)}\F_{-f_1,-f_2,f_3,f_4}(\vec B_{\vec c^{10}(f_3,f_4)}) 
\\ & +(\F_{f_1,f_2,-f_3,-f_4}(\vec{C}))_{\vec c^{00}(f_1,f_2)}\F_{f_1,f_2,f_3,f_4}(\vec B_{\vec c^{11}(f_3,f_4)}) 
\\ & +(\F_{f_1,f_2,-f_3,-f_4}(\vec{C}))_{\vec c^{01}(f_1,f_2)}\F_{f_1,-f_2,f_3,f_4}(\vec B_{\vec c^{1}(f_3,f_4)}) 
\\ & +(\F_{f_1,f_2,-f_3,-f_4}(\vec{C}))_{\vec c^{10}(f_1,f_2)}\F_{-f_1,f_2,f_3,f_4}(\vec B_{\vec c^{11}(f_3,f_4)}) 
\\ & +(\F_{f_1,f_2,-f_3,-f_4}(\vec{C}))_{\vec c^{11}(f_1,f_2)}\F_{-f_1,-f_2,f_3,f_4}(\vec B_{\vec c^{11}(f_3,f_4)}). 
\end{aligned}\end{equation}
\item The cylindrical Fourier transform \cite{BSS10} is not separable except for the case $n=2$. Here the convolution corollary holds
\begin{equation}\begin{aligned}
\F_{f_1}(\vec C_{\vec c^{0}(f_1)})\F_{f_1}(\B)+\F_{-f_1}(\vec C_{\vec c^{1}(f_1)})\F_{f_1}(\B),
\end{aligned}\end{equation}
but for all other no closed formula can be constructed in a similar way.
\end{enumerate}
\end{ex}
\section{Conclusions and Outlook}
%
In this paper we introduced the concept of coorthogonality as the property of commutation or anticommutation of blades. We proved that it is equivalent to the claim for blades to be real multiples of basis blades for an orthonormal basis and presented an algorithm to compute this basis.
\par
The Lemmata \ref{l:prodviele}, \ref{l:expcom} and \ref{l:prodtrennen}, that were primarily stated and proved in \cite{BSH11} about multiplication with invertible factors, become simplified for coorthogonal blades. We saw that in this case the partition of the multivector $\A$ takes place along the basis blades and that it is independent from the relative order of the factors it is exchanged with. A consequence of this is, that every exponential can only have four simple predictable shapes after decomposition: itself, a sine, a cosine or zero. That fact inspired the definition of the geometric trigonometric transform, whose properties will be studied in a future paper.
\par
By means of the GTT we were able to prove a convolution theorem (Theorem \ref{t:convolution}) for the general geometric Fourier transform introduced in \cite{BSH11}. It highlights the rich consequences of the geometric structure created by utilizing general geometric square roots of minus one in sets $F_1,F_2$. The information contained in the multivector fields, appears now finely segmented and related term by term. The choice of $F_1,F_2$ determines this segmentation. Because convolution appears in wavelet theory and is closely related to correlation, the GTF convolution theorem may have interesting consequences for multidimensional geometric pattern matching and neural network type learning algorithms as well as for geometric algebra wavelet theory.
%
%
 
 \addcontentsline{toc}{section}{References}
  \bibliographystyle{unsrt} 
  \bibliography{./../Literaturverzeichnis}

\end{document}